\documentclass[preprint,times]{elsarticle}
\usepackage[utf8]{inputenc}
\usepackage{amsmath}
\usepackage{amsthm}
\usepackage{bussproofs}
\usepackage{color}
\usepackage{url}

\newcommand{\Universal}{$\mathcal{U}$}

\newtheorem{definition}{Definition}[]
\newtheorem{proposition}[definition]{Proposition}
\newtheorem{lemma}[definition]{Lemma}
\newtheorem{theorem}[definition]{Theorem}
\newtheorem{remark}[definition]{Remark}
\newtheorem{corollary}[definition]{Corollary}

\newcommand{\IMPR}{$(\to_R)$}
\newcommand{\DISJL}{$(\vee_L)$}
\newcommand{\CONJR}{$(\wedge_R)$}
\newcommand{\AXL}{$(\text{Ax}_L)$}
\newcommand{\AXR}{$(\text{Ax}_R)$}
\newcommand{\CUT}{$(\text{Cut})$}
\newcommand{\MUT}{$(\tilde\mu)$}
\newcommand{\MU}{$(\mu)$}

\newcommand{\axc}[1]{\AxiomC{$#1$}}
\newcommand{\uic}[2]{\RightLabel{\scriptsize#2}\UnaryInfC{$#1$}}
\newcommand{\bic}[2]{\RightLabel{\scriptsize#2}\BinaryInfC{$#1$}}

\newcommand{\LKMMT}{$\text{LK}_{\mu\tilde\mu}$}

\newcommand{\explodes}{\Vdash_\bot}
\newcommand{\sVd}{\Vdash_{\!\!\! s}}

\newcommand{\To}{\Longrightarrow}
\newcommand{\Tocom}{\,\,\,\Longrightarrow\,\,\,}

\begin{document}

\begin{frontmatter}

\title{Kripke Models for Classical Logic}

\author[danko]{Danko Ilik}
\author[gyesik]{Gyesik Lee\corref{thanks}}
\author[hugo]{Hugo Herbelin}

\address[danko]{École Polytechnique. Address: INRIA PI.R2, 23 avenue
  d'Italie, CS 81321, 75214 Paris Cedex 13, France\\ E-mail:
  danko.ilik@polytechnique.edu}

\address[gyesik]{ROSAEC Center. Address: Bldg 138, Seoul National
  University, 171-742 Seoul, Korea.\\ E-mail: gslee@ropas.snu.ac.kr}

\address[hugo]{INRIA. Address: INRIA PI.R2, 23 avenue d'Italie, CS
  81321, 75214 Paris Cedex 13, France.\\ E-mail:
  hugo.herbelin@inria.fr}

\cortext[thanks]{For Gyesik Lee, this work was partially supported by
  the Engineering Research Center of Excellence Program of Korea
  Ministry of Education, Science and Technology(MEST) / Korea Science
  and Engineering Foundation(KOSEF), grant number
  R11-2008-007-01002-0.}

\begin{abstract}
  We introduce a notion of Kripke model for classical logic for which
  we constructively prove soundness and cut-free completeness. We
  discuss the novelty of the notion and its potential applications.
\end{abstract}

\begin{keyword}
  Kripke model \sep classical logic \sep sequent calculus \sep lambda
  mu calculus \sep classical realizability \sep normalization by
  evaluation
  
  \MSC 03F99 \sep 03H05 \sep 03B30 \sep 03B40
\end{keyword}

\end{frontmatter}

\section{Introduction}
Kripke models have been introduced as means of giving semantics to
modal logics and were later used to give semantics for intuitionistic
logic as well, cf. \cite{kripke59, kripke63}. The purpose of the present
paper is to show that Kripke models can also be used as semantics for
\emph{classical} logic.  Of course, Kripke semantics can be indirectly
assigned to classical logic by means of some appropriate
double-negation translation, as in \cite{DBLP:journals/jlp/Avigad01},
but our goal here is to provide a \emph{direct} presentation of a
notion of Kripke semantics for classical logic.

We will use the {\LKMMT} sequent calculus of \cite{CurienH00} to
represent proofs, but the conclusions given apply to any complete
formal system for classical logic. There are at least two reasons for
choosing \LKMMT: first, it is a typing system for a calculus very
close to $\lambda$-calculus and we are ultimately interested in the
computational content of classical logic; second, the symmetry of
left/right distinguished formulae of {\LKMMT} allows to give two dual
notions of models, of which only one needs to be, and is, presented in
this paper, while the other can be derived by analogy.

This paper is organised as follows. Section 2 introduces the notion of
classical Kripke model, based on two modifications to the traditional
notion, and discusses the relationship between the traditional and our
notion. Section 3 introduces the sequent calculus {\LKMMT} and gives a
soundness theorem for it. Section 4 proves a completeness theorem for
a universal model constructed from the deduction system
itself. Section 5 is the concluding section which discusses related
and future work.

We use the standard inductive definition of predicate logic formulae
for the connectives
$\{\top,\bot,\wedge,\vee,\to,\exists,\forall\}$. The language has
infinitely many constants. A sentence is a formula where all variables
are bound by quantifiers. An atomic formula is one which is not built
up from logical connectives, i.e. it is one built up of a predicate
symbol. The shorthand $\neg A$ stands for $A\to \bot$.

All statements and proofs are constructive.

\section{Classical Kripke Models}

Kripke models can be considered as the ``most classical'' of all the
semantics for intuitionistic logic, for two reasons: first, each of
the `possible worlds' that define a Kripke model is a classical world
in itself (where either an atom or its negation are true); second, it
is the single of the semantics for intuitionistic logic which has only
a classical proof of completeness, when disjunction and existential
quantification are considered.\footnote{ There is an intuitionistic
  proof in \cite{Veldman76}, but it makes use of the fan theorem which
  is not universally recognised as constructive.}

In the last two decades, the Curry-Howard correspondence between
intuitionistic proof systems and typed lambda-calculi has been
extended to classical proof systems
\cite{Griffin90,Parigot92,CurienH00}. The idea for introducing
direct-style Kripke models for classical logic came from their
usefulness in providing normalisation-by-evaluation for intuitionistic
proof systems \cite{cCoquand93,CCoquand2002}. To account for a
classical proof system we modify the traditional notion of Kripke
model in the following two ways.

\paragraph{Not taking the forcing relation as primitive} We take as
primitive the notion of ``strong refutation'', and define forcing in
terms of it.\footnote{For an alternative, see the discussion on dual
  models in Section 5.} The forcing definition we get in this way
partly coincides with the traditional definition of forcing, as
explained in subsection~\ref{traditional_forcing}.

\paragraph{Allowing certain nodes to validate absurdity} We allow
certain possible worlds to be marked as ``fallible'', or
``exploding''. This approach has been taken for Kripke models in
\cite{Veldman76}, for Beth models by Friedman \cite{Troelstra} and is
necessary in order to have a constructive proof of completeness, in
the view of the meta-mathematical results from \cite{Kreisel62,
  McCarty94, McCarty02}, which preclude constructive
proofs\footnote{Strictly speaking, the cited results show that having
  a constructive proof of completeness implies having a proof of
  Markov's Principle.} of completeness in case one wants to retain
that absurdity must never be valid in a possible
world\footnote{Extending the class of Boolean models with inconsistent
  models is also the key to the constructive proof of the classical
  completeness theorem in \cite{Krivine96}. For an analysis of that
  result, see \cite{BerardiV04}.}.

\begin{definition} A \emph{classical Kripke model} is given by a
  quintuple $(K, \le , D, \sVd, \Vdash_\bot)$, $K$ inhabited, such
  that
  \begin{itemize}
  \item $(K, \le)$ is a poset of ``possible worlds''; 

  \item $D$ is the ``domain function'' assigning sets to the
    elements of $K$ such that
    \[ \forall w, w' \in K, (w \le w' \Rightarrow D(w) \subseteq
    D(w'))\] i.e., $D$ is monotone;

    \noindent Let the language be extended with constant symbols for
    each element of ${\cal D} := \cup\{D(w) : w\in K \}$.

  \item $(-): (-)\sVd$ is a binary relation of ``strong refutation''
    between worlds and atomic sentences in the extended language such
    that
    \begin{itemize}
    \item $w : X(d_1, ..., d_n) \sVd\quad \Rightarrow\quad d_i \in
      D(w) $ \, for each $i \in \{1, ..., n \}$,
    \item (Monotonicity) $w : X(d_1, ..., d_n) \sVd\,\, \&\,\, w \le
      w' \,\, \Rightarrow\,\, w' : X(d_1,...,d_n) \sVd$,
    \end{itemize}
    
  \item $(-) \Vdash_\bot$ is a unary relation on worlds labelling a
    world as ``exploding'', which is also monotone:
    \[
    w\explodes \&~ w\le w' \Rightarrow w'\explodes.
    \]
\end{itemize}

\end{definition}
The strong refutation relation is extended from atomic to composite
sentences inductively and by mutually defining the relations of
\emph{forcing} and (non-strong) \emph{refutation}.

\begin{definition}\label{def-composite} 
  The relation $(-) : (-) \sVd$ of \emph{strong refutation is extended
    to} the relation between worlds $w$ and \emph{composite sentences}
  $A$ in the extended language with constants in $D(w)$, inductively,
  together with the two new relations:
  \begin{itemize}
  \item A sentence $A$ is \emph{forced} in the world $w$ (notation
    $w:\Vdash A$) if any world $w'\ge w$, which strongly refutes $A$,
    is exploding;

  \item A sentence $A$ is \emph{refuted} in the world $w$ (notation
    $w:A\Vdash$) if any world $w'\ge w$, which forces $A$, is
    exploding;
  \end{itemize}
  \begin{itemize}
  \item $w:A\wedge B\sVd$ if $w:A\Vdash$ or $w:B\Vdash$;
  \item $w:A\vee B\sVd$ if $w:A\Vdash$ and $w:B\Vdash$;
  \item $w:A\to B\sVd$ if $w:\Vdash A$ and $w:B\Vdash$;
  \item $w:\forall x. A(x)\sVd$ if $w:A(d)\Vdash$ for some $d\in D(w)$;
  \item $w:\exists x. A(x)\sVd$ if, for any $w'\ge w$ and $d \in
    D(w')$, $w':A(d)\Vdash$;
  \item $\bot$ is always strongly refuted;
  \item $\top$ is never strongly refuted.
  \end{itemize}
\end{definition}

The notions of forcing and refutation can be somewhat understood as
the classical notions of being true and being false. However, a
statement of form $P \Rightarrow w\Vdash_\bot$ should not be thought
of as negation of $P$ at the meta-level, because in the concrete model
we provide in section \ref{completeness}, $w\Vdash_\bot$ is always an
inhabited set. In other words, we never use \textit{ex falso
  quodlibet} at the meta-level to handle exploding nodes.

The notion of strong refutation is more informative than the notion of
(non-strong) refutation, not only because the former implies the
latter, but also because, for example, having $w:A\wedge B\sVd$ tells
us which one of $A, B$ is refuted, while $w:A\wedge B\Vdash$ does not.

A more detailed characterisation of the notions is given in the rest
of this section.

\begin{lemma}\label{monotone}
  Strong refutation, forcing and refutation are monotone in any
  classical Kripke model.
\end{lemma}
\begin{proof}
  The monotonicity of strong refutation can be proved by induction on
  the formula in question, while that of forcing and refutation is
  obviously true.
\end{proof}

\begin{lemma}\label{srefutes_refutes} 
  Strong refutation implies refutation: In any world $w$ and for any
  sentence $A$, $w:A\sVd$ implies $w:A\Vdash$.
\end{lemma}
\begin{proof} 
  Suppose $w :A \sVd$, $w' \ge w$ and $w' : \Vdash A$. Then $w'$ is
  exploding because $w' :A \sVd$ by monotonicity. Since $w'$ was
  arbitrary, $w :A \Vdash$.
\end{proof}

\subsection{Relation to Traditional Forcing and Further
  Properties}\label{traditional_forcing}
It is natural to ask what is the relationship between traditional
intuitionistic forcing\cite{Troelstra} and our forcing whose
definition relies on a more primitive notion.  Lemmas
\ref{forcing_comparison} and \ref{top-bottom} give that the two
notions (superficially) coincide on the fragment of formulae
constructed by $\{\to,\wedge,\forall,\top\}$

\begin{lemma}\label{forcing_comparison} The following statements hold.
  \begin{eqnarray}
    w:\Vdash A\to B & \Longleftrightarrow & \text{ for all } w'\ge w,
    w':\Vdash A \Rightarrow w':\Vdash B \label{imply}\\
    w:\Vdash A\land B & \Longleftrightarrow & w:\Vdash A \text{ and } w:\Vdash B \label{and}\\
    w:\Vdash \forall x.A(x) & \Longleftrightarrow & \text{ for all }
    w'\ge w \text{ and } d\in D(w'), w':\Vdash A(d) \label{forall}\\
    w :\Vdash A \lor B & \Longleftarrow & w:\Vdash A \text{ or } w :\Vdash B \\
    w:\Vdash \exists x.A(x) & \Longleftarrow & \text{ for some } d\in D(w), w:\Vdash A(d)
  \end{eqnarray}
\end{lemma}
\begin{proof}
  Lemma \ref{monotone} and Lemma \ref{srefutes_refutes} are used
  implicitly in the following proof.
  \begin{itemize}
  \item[(\ref{imply})] Left-to-right: Suppose $w'\ge w$ and $w':\Vdash
    A$. To show $w':\Vdash B$ we let $w'' \ge w'$ and $w'':B\sVd$ and
    have to show that $w''$ is exploding. Since then $w'' : A \to B
    \sVd$ holds by monotonicity and Lemma \ref{srefutes_refutes}, the
    claim follows from the definition of $w:\Vdash A\to B$.

    Right-to-left: Suppose $w'\ge w$ and $w' :A\to B \sVd$, i.e.,
    $w':\Vdash A$ and $w':B\Vdash$. We have to show $w'$ is
    exploding. But, this is immediate, since $w' :\Vdash B$ by
    assumption.

  \item[(\ref{and})] Left-to-right: Suppose $w' \ge w$ and $w' :A
    \sVd$. Then $w' : A \Vdash$, and so $w' : A \land B \sVd$. This
    implies that $w'$ is exploding, that is, $w :\Vdash A$. Similarly,
    we can show $w :\Vdash B$.

    Right-to-left: Suppose $w' \ge w$ and $w' : A \land B
    \sVd$. Therefore we have $w' : A \Vdash$ or $w' : B \Vdash$. Each
    case leads to $w' :\Vdash_\bot$ since $w' :\Vdash A$ and $w'
    :\Vdash B$ by monotonicity.

  \item[(\ref{forall})] Left-to-right: Suppose $w''\ge w'\ge w$, $d
    \in D(w')$, and $w'' : A(d) \sVd$. Then $w'' :\forall x. A(x)
    \sVd$, so $w''$ is exploding.

    Right-to-left: Suppose $w' \ge w$ and $w' : \forall x. A(x) \sVd$,
    i.e., $w' : A(d) \Vdash$ for some $d \in D(w')$. So $w'$ is
    exploding by assumption.
  \end{itemize}
  The rest of the cases are obvious.
\end{proof}

Note, however, that although the definitions of our and intuitionistic
forcing ``match'' on the fragment $\{\to,\wedge,\forall,\top\}$, that
does not mean that a formula in that fragment is forced in our sense
if and only if it is forced in the intuitionistic sense. The law of
Peirce $((A\to B)\to A)\to A$ is one counterexample to that, it is
classically but not intuitionistically forced; this is so because in
our forcing, hidden under the surface, there is a notion of refutation
which can be used.

\begin{remark}\label{or-exists1}
  The following do not hold in general, even if reasoning classically.
  \begin{itemize}
  \item $ w :\Vdash A \lor B \Longrightarrow w:\Vdash A \text{ or } w
    :\Vdash B$.
  \item $ w:\Vdash \exists x.A(x) \Longrightarrow \text{ for some }
    t\in D(w), w:\Vdash A(t)$.
  \end{itemize}
  The explanation is deferred to Remark \ref{or-exists2}.
\end{remark}

\begin{lemma}\label{cbn-to-cbv}
  Given a classical Kripke model $\cal K$, the following hold.
  \begin{enumerate}
  \item $w : A \to B \Vdash$ \, iff $w : A \to B \sVd$. \label{cbn-to}

  \item $w : A \lor B \Vdash$ \, iff\, $w : A \lor B
    \sVd$.\label{cbn-or}

  \item $w: \exists x.A(x)\Vdash$ \, iff\, $w : \exists x. A(x)
    \sVd$. \label{9-6'}

  \item If $w : A \Vdash$ or $w : B \Vdash$, then $w : A \land B
    \Vdash$.\label{cbn-and}

  \item If $w : A(d) \Vdash$ for some $d \in D(w)$, then $w : \forall
    x.A(x) \Vdash$.\label{cbn-all}
  \end{enumerate}
\end{lemma}
\begin{proof}
  \begin{enumerate}
  \item Right-to-left is Lemma \ref{srefutes_refutes}.

    Left-to-right: Suppose $w' \ge w$ and $w' : A \sVd$. In order to
    show that $w'$ is exploding it suffices to show $w' :\Vdash A \to
    B$. For this assume $w'' \ge w'$ and $w'' :A \to B \sVd$, i.e.,
    $w'' :\Vdash A$ and $w'' :B \Vdash$. Then $w''$ is exploding since
    we have $w'' :A\sVd$ by monotonicity. Similarly, we can show
    $w:B\Vdash$.

  \item Right-to-left is Lemma \ref{srefutes_refutes}.

    Left-to-right: Suppose $w' \ge w$ and $w' : \Vdash A$. Then by
    Lemma \ref{forcing_comparison}, $w' :\Vdash A \lor B$ holds. So
    $w'$ is exploding. That is $w : A \Vdash$. Similarly, $w :B
    \Vdash$ holds.

  \item Right-to-left is Lemma \ref{srefutes_refutes}.

    Left-to-right: Suppose $w'' \ge w'\ge w$, $d\in D(w')$ and $w''
    :\Vdash A(d)$. Then by Lemma \ref{forcing_comparison}, $w''
    :\Vdash \exists x. A(x)$. So $w''$ is exploding since we have $w''
    : \exists x. A(x) \Vdash$ by monotonicity.

  \item Suppose w.l.o.g. $w : A \Vdash$, $w' \ge w$ and $w' :\Vdash A
    \land B$. Then by Lemma \ref{forcing_comparison}, $w' :\Vdash
    A$. So $w'$ is exploding because we have $w' : A \Vdash$ by
    monotonicity.

  \item Suppose $w' \ge w$ and $w' :\Vdash \forall x.A(x)$. Then by
    Lemma \ref{forcing_comparison}, $w' :\Vdash A(d)$. So $w'$ is
    exploding because we have $w' : A(d) \Vdash$ by monotonicity.
  \end{enumerate}
\end{proof}


We can also say that forcing of $\bot$ and $\top$ behaves like
expected with respect to exploding nodes \cite{Veldman76,Krivine96}:

\begin{lemma}\label{top-bottom}
  \begin{enumerate}
  \item $w :\Vdash \top$ and $w : \bot \Vdash$.

  \item $w$ is exploding iff $w :\Vdash \bot$.

  \item $w$ is exploding iff $w : \top \Vdash$.
  \end{enumerate}
\end{lemma}
\begin{proof}
  \begin{enumerate}
  \item Obvious.

  \item Let $w$ be an arbitrary world.
    \begin{eqnarray*}
      w :\Vdash \bot & \Longleftrightarrow & \forall (w' \ge w)\, (w' : \bot \sVd\,\, \Rightarrow\,\, w' :\Vdash_\bot) \\
      & \Longleftrightarrow & \forall (w' \ge w) \, (w' :\Vdash_\bot) \,\,\iff\,\,  w :\Vdash_\bot
    \end{eqnarray*}

  \item Similar.
  \end{enumerate}
\end{proof}

We can use the previous lemmas to show that the forcing relation for
classical logic behaves ``classically'' indeed:

\begin{lemma}\label{lemma-9}
  The following hold in the classical Kripke semantics.
  \begin{enumerate}
\item $w:\Vdash A \iff w:\neg A\sVd$. \label{9-1}
\item $w : A \Vdash \iff w :\Vdash \neg A$. \label{9-2}
\item $w : \neg A \Vdash \iff w :\Vdash A$. \label{9-2'}
\item $w : \neg A \Vdash \iff w :\neg A\sVd$. \label{9-2''}
\item $w : \Vdash A \iff w :\Vdash \neg\neg A$. \label{9-3}
\item $w : A \Vdash \iff w : \neg\neg A \Vdash$. \label{9-4}
\item $w : \neg A \sVd \iff w : \Vdash \neg \neg A \Vdash \iff w : \Vdash A$.\label{9-x}
\end{enumerate}
\end{lemma}
\begin{proof}
  \begin{enumerate}
  \item Obvious by definition because $w:\bot \Vdash$.

  \item It follows from Lemma \ref{forcing_comparison}.

  \item Obvious by Lemma \ref{cbn-to-cbv} and the previous claims.

  \item[\ref{9-2''}.] $\sim$ \ref{9-x}. Obvious from the previous
    claims.

\end{enumerate}
\end{proof}

\begin{corollary}\label{double-neg-refutation}
  In any classical Kripke model, the following holds.
  \begin{equation*}
    w : \neg A \sVd \iff w : \Vdash \neg \neg A \iff w : \Vdash A
  \end{equation*}
\end{corollary}

We now consider the following double-negation translation $(\cdot)^*$,
which is the one of G\"odel-Gentzen\cite{goedel33,gentzen36}, except
that atomic formulae, $\bot$ and $\top$ are not doubly negated:
\begin{eqnarray*}
X^* & := & X \quad \text{($X$ is atomic, $\bot$ or $\top$)}\\
(A \land B)^* & := & A^* \land B^*\\
(A \to B)^* & := & A^* \to B^*\\
(\forall x. A)^* & := & \forall x. A^*\\
(A \lor B)^* & := & \neg (\neg A^* \land \neg B^*)\\
(\exists x.A)^* & := & \neg \forall x.\neg A^*
\end{eqnarray*}

\begin{proposition} \label{DN-proposition}
  Every classical Kripke model $\mathcal{C} =
  (K,\le,D,\sVd,\Vdash_\bot)$ gives rise to an intuitionistic Kripke
  model with exploding worlds $\mathcal{I} =
  (K,\le,D,\Vdash_i,\Vdash_\bot)$, which inherits all components of
  $\mathcal{C}$, except for $\Vdash_i$, which is defined for atomic
  formulae by non-strong forcing, i.e.
  \[
  w\Vdash_i X \text{ iff } w:\Vdash X
  \]
  The translation $(\cdot)^*$ relates $\mathcal{C}$ and $\mathcal{I}$,
  that is, for any world $w$ and any formula $A$ we have
  \[
  w\Vdash_i A^* \text{ iff } w:\Vdash A.
  \]
\end{proposition}
\begin{proof} By induction on the complexity of $A$ and by using
  (1)-(3) from Lemma \ref{forcing_comparison} and (2) from Lemma
  \ref{top-bottom}. We detail only the induction case for $\vee$,
  which is the most involved one:
  \begin{eqnarray*}
    & w\Vdash_i (A\vee B)^* & \iff \\
    & w\Vdash_i \neg(\neg A^* \wedge \neg B^*) & \iff \\
    (\forall w'\ge w) & 
    [ w'\Vdash_i\neg A^*, w'\Vdash_i \neg B^* \To w'\Vdash_i\bot] & \iff \\
    (\forall w'\ge w)[ & 
    (\forall w''\ge w')[w''\Vdash_i A^* \To w''\Vdash_i\bot], & \\
    &
    (\forall w''\ge w')[w''\Vdash_i B^* \To w''\Vdash_i\bot] & \\
    & \To w'\Vdash_i\bot] & \iff \\
    (\forall w'\ge w)[ & 
    (\forall w''\ge w')[w'':\Vdash A \To w''\Vdash_\bot], & \\
    &
    (\forall w''\ge w')[w'':\Vdash B \To w''\Vdash_\bot] & \\
    & \To w'\Vdash_\bot] & \iff \\
    (\forall w'\ge w) & 
    [ w':A\Vdash, w':B\Vdash \To w'\Vdash_\bot] & \iff \\
    (\forall w'\ge w) & 
    [ w':A\vee B\sVd \To w'\Vdash_\bot] & \iff \\
    & w:\Vdash A\vee B & 
  \end{eqnarray*}
\end{proof}

\section{{\LKMMT} and Soundness}\label{soundness}

\begin{table}[t]
  \begin{tabular}{ | c | c |}
    \hline
    ~ & ~ \\
    \AxiomC{}
    \RightLabel{\AXL}
    \UnaryInfC{$\Gamma|A\vdash A,\Delta$}
    \DisplayProof
    &
    \AxiomC{}
    \RightLabel{\AXR}
    \UnaryInfC{$A,\Gamma\vdash A|\Delta$}
    \DisplayProof
    \\
    ~ & ~ \\
    \AxiomC{$\Gamma,A\vdash\Delta$}
    \RightLabel{$(\tilde\mu)$}
    \UnaryInfC{$\Gamma|A\vdash\Delta$}
    \DisplayProof
    &
    \AxiomC{$\Gamma\vdash A,\Delta$}
    \RightLabel{$(\mu)$}
    \UnaryInfC{$\Gamma\vdash A|\Delta$}
    \DisplayProof
    \\
    ~ & ~ \\
    \AxiomC{$\Gamma\vdash A|\Delta$} \AxiomC{$\Gamma|B\vdash \Delta$}
    \RightLabel{$(\to_L)$}
    \BinaryInfC{$\Gamma|A\to B\vdash \Delta$}
    \DisplayProof
    &
    \AxiomC{$\Gamma,A\vdash B|\Delta$}
    \RightLabel{$(\to_R)$}
    \UnaryInfC{$\Gamma\vdash A\to B|\Delta$}
    \DisplayProof
    \\
    ~ & ~ \\
    \AxiomC{$\Gamma|A\vdash\Delta$}\AxiomC{$\Gamma|B\vdash\Delta$}
    \RightLabel{$(\vee_L)$}
    \BinaryInfC{$\Gamma|A\vee B\vdash\Delta$}
    \DisplayProof
    &
    \AxiomC{$\Gamma\vdash A|\Delta$}
    \RightLabel{$(\vee^1_R)$}
    \UnaryInfC{$\Gamma\vdash A\vee B|\Delta$}
    \DisplayProof
    \quad
    \AxiomC{$\Gamma\vdash B|\Delta$}
    \RightLabel{$(\vee^2_R)$}
    \UnaryInfC{$\Gamma\vdash A\vee B|\Delta$}
    \DisplayProof
    \\
    ~ & ~ \\
    \AxiomC{$\Gamma|A\vdash \Delta$}
    \RightLabel{$(\wedge^1_L)$}
    \UnaryInfC{$\Gamma|A\wedge B\vdash \Delta$}
    \DisplayProof
    \quad
    \AxiomC{$\Gamma|B\vdash \Delta$}
    \RightLabel{$(\wedge^2_L)$}
    \UnaryInfC{$\Gamma|A\wedge B\vdash \Delta$}
    \DisplayProof
    &
    \AxiomC{$\Gamma\vdash A|\Delta$}\AxiomC{$\Gamma\vdash B|\Delta$}
    \RightLabel{$(\wedge_R)$}
    \BinaryInfC{$\Gamma\vdash A\wedge B|\Delta$}
    \DisplayProof
    \\
    ~ & ~ \\
    \AxiomC{$\Gamma|A(x)\vdash\Delta$}\AxiomC{$x$ fresh}
    \RightLabel{$(\exists_L)$}
    \BinaryInfC{$\Gamma|\exists x A(x)\vdash\Delta$}
    \DisplayProof 
    & 
    \AxiomC{$\Gamma\vdash A(t)|\Delta$}
    \RightLabel{$(\exists_R)$}
    \UnaryInfC{$\Gamma\vdash \exists x. A(x)|\Delta$}
    \DisplayProof
    \\
    ~ & ~ \\
    \AxiomC{$\Gamma|A(t)\vdash \Delta$}
    \RightLabel{$(\forall_L)$}
    \UnaryInfC{$\Gamma|\forall x. A(x)\vdash \Delta$}
    \DisplayProof
    & 
    \AxiomC{$\Gamma\vdash A(x)|\Delta$}\AxiomC{$x$ fresh}
    \RightLabel{$(\forall_R)$}
    \BinaryInfC{$\Gamma\vdash \forall x A(x)|\Delta$}
    \DisplayProof 
    \\
    ~ & ~ \\
    \AxiomC{}
    \RightLabel{$(\bot_L)$}
    \UnaryInfC{$\Gamma|\bot\vdash\Delta$}
    \DisplayProof
    & ~
    \\
    ~ & ~ \\
    ~
    &
    \AxiomC{}
    \RightLabel{$(\top_R)$}
    \UnaryInfC{$\Gamma\vdash\top|\Delta$}
    \DisplayProof
    \\
    ~ & ~ \\
    \hline
    \multicolumn{2}{ |c|}{~}\\
    \multicolumn{2}{ |c|}{
      \AxiomC{$\Gamma\vdash A|\Delta$}\AxiomC{$\Gamma|A\vdash\Delta$}
      \RightLabel{\CUT}
      \BinaryInfC{$\Gamma\vdash\Delta$}
      \DisplayProof 
    }\\
    \multicolumn{2}{ |c|}{~}\\
    \hline
  \end{tabular}
  \caption{The sequent calculus \LKMMT}
  \label{LKmmt}
\end{table}

To emphasise the symmetries of classical logic, we use a sequent
calculus in the style of Gentzen's LK as proof system. We could have
directly used LK or one of its variants with implicit structural
rules, {\em \`a la} Kleene-Kanger. In practise, even though the
current paper does not go into the details of the computational
content of proofs, we rely here on {\LKMMT} which has a simple symmetrical
variant of $\lambda$-calculus as underlying language of
proofs~\cite{CurienH00,HerbelinHabilitation}\footnote{Note that even
  if not based on $\lambda$-calculus, there are calculi of proof-terms
  for LK too, see
  e.g.~\cite{UrbanPhD,DBLP:journals/entcs/Lengrand03,DBLP:conf/ictcs/BakelLL05}.}.

{\LKMMT} is presented on Table~\ref{LKmmt}. It differs from LK in the
following points:

\begin{itemize}

\item Sequents come with an explicitly distinguished formula on the
  right or on the left, or no distinguished formula at all, resulting
  in three kinds of sequents: ``$\Gamma \vdash \Delta$'', ``$\Gamma |
  A \vdash \Delta$'' and ''$\Gamma \vdash A | \Delta$''. Especially,
  the distinguished formula plays an ``active'' r\^{o}le in the
  rules.

\item Accordingly, the axiom rule splits into two variants {\AXL} and
  {\AXR} depending on whether the left active formula or the right
  active formula is distinguished. There are also two new rules,
  $(\mu)$ and $(\tilde\mu)$, for making a formula active\footnote{Note
    that we have to define the contexts of formulae $\Gamma$ and
    $\Delta$ as ordered sequences to get a non ambiguous
    interpretation of {\LKMMT} as a typed $\lambda$-calculus.}.

\item There are no explicit contraction rules: contractions are
  derivable from a cut against an axiom as follows:

  \begin{itemize}
  \item Left contraction:
    \begin{equation*}\tag{$Contr_L$}
      \begin{tabular}{cc}
        \axc{}
        \uic{\Gamma,A\vdash A\mid \Delta}{\AXR}
        \axc{\Gamma,A\mid A\vdash\Delta}
        \bic{\Gamma,A\vdash\Delta}{\CUT}
        \DisplayProof
      \end{tabular}
    \end{equation*}

\item Right contraction:
  \begin{equation*}\tag{$Contr_R$}
    \begin{tabular}{cc}
      \axc{\Gamma\vdash A\mid A,\Delta}
      \axc{}
      \uic{\Gamma\mid A\vdash A,\Delta}{\AXL}
      \bic{\Gamma\vdash A,\Delta}{\CUT}
      \DisplayProof
    \end{tabular}
  \end{equation*}
\end{itemize}

\item Consequently, the notion of normal proof, or cut-freeness, is
  slightly different from the notion of cut-freeness in LK: a {\it
    normal proof} is a proof whose only cuts are of the form of a cut
  between an axiom and an introduction rule\footnote{The rules \MU\
    and \MUT\ are not introduction rules, because they do not
    introduce a formula constructor.}. This is the notion that we
  refer to when below, very often, we say ``cut-free'' or ``provable
  without a cut''.

\end{itemize}

The correspondence between normal proofs of LK and normal proofs of
{\LKMMT} is direct. If we present LK with weakening rules attached to
the axiom rules \textit{\`a la} Kleene's $G_4$ or Kanger's LC, 
we obtain an LK proof
from an {\LKMMT} proof by erasing the bars serving to distinguish active
formulae, and by removing the trivial inferences coming from the rules
$(\mu)$ and $(\tilde\mu)$. In the other way round, every introduction
rule of LK can be derived in {\LKMMT} by applying the rules $(\mu)$ and
$(\tilde\mu)$ on the premises and a (possibly dummy) contraction
(i.e. a cut against an axiom) on the conclusion of the rule. Similarly
for the axiom rule (for which there are two possible derivations) and
the cut rule. For more details we refer the reader to
\cite{CurienH00}.



For a constant $c$, let $\Gamma_c(t), \Delta_c(t), A_c(t)$ be obtained
from $\Gamma, \Delta, A$ by replacing each constant $c$ with a term
$t$.

\begin{lemma}[Weakening]\label{weakening}
  Suppose $\Gamma \subseteq \Gamma'$ and $\Delta \subseteq\Delta'$.
  \begin{itemize}
  \item $\Gamma\vdash\Delta$ implies $\Gamma'\vdash\Delta'$.
    
  \item $\Gamma\vdash A\mid \Delta$ implies $\Gamma'\vdash A\mid \Delta'$.

  \item $\Gamma\mid A\vdash\Delta$ implies $\Gamma'\mid A\vdash\Delta'$.
  \end{itemize}
  Moreover, no further cuts in the derivations on the right-hand side
  are necessary.
\end{lemma}

\begin{lemma}\label{fresh-general}
  Let $c$ be a constant and $y$ a variable which does not appear in
  $\Gamma, \Delta, A$.
  \begin{itemize}
  \item $\Gamma\vdash \Delta$ implies $\Gamma_c(y)\vdash \Delta_c(y)$.

  \item $\Gamma\vdash A\mid \Delta$ implies $\Gamma_c(y)\vdash
    A_c(y)\mid \Delta_c(y)$.

  \item $\Gamma \mid A\vdash\Delta$ implies $\Gamma_c(y) \mid
    A_c(y)\vdash \Delta_c(y)$.
  \end{itemize}
  Moreover, no further cuts in the derivations on the right-hand side
  are necessary.
\end{lemma}

The following lemma says that a fresh constant is as good as a fresh
variable and will play an important role in the proof of cut-free
completeness below.
\begin{lemma}[Fresh constants]\label{fresh}
  Let $c$ be a constant and $y$ a variable which does not appear in
  $\Gamma, \Delta, A$. Assume furthermore that $c$ does not appear in $\Gamma, \Delta$.
  \begin{itemize}
  \item $\Gamma\vdash A(c)\mid \Delta$ implies $\Gamma\vdash A(y)\mid \Delta$.

  \item $\Gamma \mid A(c)\vdash\Delta$ implies $\Gamma \mid A(y)\vdash \Delta$.
  \end{itemize}
  Moreover, no further cuts in the derivations on the right-hand side
  are necessary.
\end{lemma}
\begin{proof}
  It follows directly from the lemma just before.
\end{proof}

The fact that Lemma \ref{weakening} $\sim$ Lemma \ref{fresh} need not
introduce any new cuts in the derivations on the right-hand side of
the implication will be important for the proof of cut-free
completeness.

\newcommand{\dom}{\mathrm{\textsf{dom}}}

We now show the soundness of {\LKMMT} with respect to the Kripke
semantics. First we need some preparations.

Let $(K, \le , D, \sVd, \Vdash_\bot)$ be a Kripke
model. \emph{Associations} are functions from a finite set of free
variables to $\bigcup_{w\in K} D(w)$. The letters $\rho, \eta, ...$
vary over associations. Given an association $\rho$ and a free
variable $x$, $\rho^{-x}$ denotes the function obtained from $\rho$ by
deleting $x$ from its domain, i.e., $\dom(\rho^{-x}) = \dom(\rho)
\backslash \{x\}$. Let $\rho(x\mapsto d)$ denote the function $\rho'$
such that $\rho'(y) = \rho (y)$ if $y \neq x$ and $d$ otherwise.

Let $c_0$ be a distinguished constant of the language. Given a formula
$A$, let $A[\rho]$ denote the sentence in the extended language with
fresh constants for each element of $D$ obtained from $A$ by replacing
each free variable $x$ with $\rho(x)$ if $x \in \dom(\rho)$ and with
$c_0$ otherwise. $\Gamma[\rho]$ is the context obtained from $\Gamma$
by replacing each $A \in \Gamma$ with $A[\rho]$.

We write $w:\Vdash \Gamma$ when $w$ forces all sentences from $\Gamma$ and
$w:\Delta\Vdash$ when $w$ refutes all sentences from $\Delta$.

The intuitive meaning of the following theorem is that if every
formula in the assumption is forced, then not all formulae in the
conclusion can be refuted.

\begin{theorem}[Soundness] Let A be a formula and $\Gamma, \Delta$
  contexts of formulae. In any classical Kripke model $(K, \le , D,
  \sVd, \Vdash_\bot)$ the following holds: Let $w \in K$ and $\rho$ be
  an associations with the values from $D(w)$.
  \begin{itemize}
  \item If $\Gamma\vdash \Delta$, $w:\Vdash \Gamma [\rho]$ and
    $w:\Delta[\rho]\Vdash$, then $w:\Vdash_\bot$.

  \item If $\Gamma\vdash A | \Delta$, $w:\Vdash \Gamma[\rho]$ and
    $w:\Delta[\rho]\Vdash$, then $w:\Vdash A[\rho]$.

  \item If $\Gamma | A \vdash \Delta$, $w:\Vdash \Gamma[\rho]$ and
    $w:\Delta[\rho]\Vdash$, then $w:A[\rho] \Vdash$.
  \end{itemize}
\end{theorem}
\begin{proof} One proves easily the three statements simultaneously by
  induction on the derivations. We demonstrate two non-trivial
  cases. Suppose $w :\Vdash \Gamma [\rho]$ and $w :\Delta[\rho]
  \Vdash$.

  \begin{itemize}
  \item Case $(\lor_L)$: Suppose $w' \ge w$ and $w' :\Vdash A[\rho]
    \lor B[\rho]$. We have to show $w'$ is exploding. But this follows
    from the fact that $w' : A[\rho] \lor B[\rho] \sVd$. Note just
    that $w' : A[\rho]\Vdash$ and $w' :B [\rho]\Vdash$ follow from the
    I.H. using monotonicity.

  \item Case $(\exists_L)$: Suppose $w' \ge w$ and $w' :\Vdash
    (\exists x. A)[\rho]$. We have to show $w'$ is exploding. For this
    it suffices to show $w' : (\exists x. A(x)) [\rho] \sVd$, i.e.,
    $w'' : A[\rho(x\mapsto d)]) \Vdash$ for all $w'' \ge w'$ and $d\in D(w'')$. Note first that $w'' :\Vdash \Gamma [\rho(x\mapsto d)]$ and $w''
    :\Delta[\rho(x\mapsto d)] \Vdash$ by monotonicity because of the
    freshness of $x$. By I.H. the claim follows.
  \end{itemize}
\end{proof}

\section{Completeness}\label{completeness}

As usual when constructively proving completeness of Kripke semantics
for a fragment\footnote{As previously remarked, there is no
  constructive proof for full intuitionistic predicate logic.} of
intuitionistic logic \cite{cCoquand93, HerbelinLee, sozeaua}, we
define a special purpose model, called the \emph{universal model},
built from the deduction system itself. Once we show completeness for
this special model, completeness for any model follows
(Corollary~\ref{usual_completeness}).

\begin{definition}The \emph{Universal classical Kripke model}
  \Universal \ is obtained by setting:
  \begin{itemize}
  \item $K$ to the set of pairs $(\Gamma,\Delta)$ of contexts of
    \LKMMT;
  \item $(\Gamma,\Delta)\le(\Gamma',\Delta')$ iff both
    $\Gamma\subseteq\Gamma'$ and $\Delta\subseteq\Delta'$;
  \item $(\Gamma,\Delta):X\sVd$ iff the sequent $\Gamma|X\vdash\Delta$
    is provable without a cut in \LKMMT;
  \item $(\Gamma,\Delta):\Vdash_\bot$ iff the sequent
    $\Gamma\vdash\Delta$ is provable without a cut in \LKMMT;
  \item for any $w$, $D(w)$ is the set of closed terms of
    \LKMMT.
  \end{itemize}
  Note that the domain function $D$ is a constant function, while in
  the abstract definition of model we allow for non-constant domain
  functions because that allows building more counter-models in
  applications.
\end{definition}

Monotonicity of strong refutation on atoms follows from Lemma
\ref{weakening}.

\setcounter{equation}{0}

\begin{theorem}[Cut-Free Completeness for \Universal]\label{Ucompleteness}
  For any sentence $A$ and contexts of sentences $\Gamma$ and
  $\Delta$, the following hold in \Universal:
  \begin{eqnarray}
    (\Gamma,\Delta):\Vdash A & \Longrightarrow & \Gamma\vdash A|\Delta \label{forces}\\
    (\Gamma,\Delta):A\Vdash & \Longrightarrow & \Gamma|A\vdash \Delta\label{refutes}
  \end{eqnarray}
  Moreover, the derivations on the right-hand side of $(\ref{forces})$
  and $(\ref{refutes})$ are cut-free.
\end{theorem}
\begin{proof} 
  We proceed by simultaneously proving the two statements by induction
  on the complexity of $A$. When quantifiers are concerned, $A(t)$ has
  lower complexity than $\exists x. A(x)$ and $\forall x. A(x)$.

  The derivation trees in this proof use meta-rules (*) and multi-step
  derivations ($Contr_L, Contr_L$) in addition to the derivation rules
  of the calculus from Table~\ref{LKmmt} in order to make the proofs
  easier to read.

  We also remind the reader that the notion of cut-freeness is the one
  of \LKMMT, introduced in the previous section.
  
  \renewcommand{\labelenumi}{(\arabic{enumi}) }
  \paragraph{Base case for atomic formulae}
  In the base case we have forcing and refutation on atomic sentences,
  which by definition reduce to strong refutation on atomic sentences,
  which by definition reduces just to statements about the deductions
  in {\LKMMT}.
  \begin{enumerate}
  \item[$(\ref{forces})$] Suppose
    \[
    \forall (\Gamma',\Delta')\ge (\Gamma,\Delta),
    \{\Gamma'|X\vdash\Delta' \Longrightarrow \Gamma'\vdash\Delta'\}
    \tag{*} \] where the RHS is cut-free. Then the following holds for
    $\Gamma' = \Gamma$ and $\Delta' = X, \Delta$:
    \begin{prooftree}
      \axc{}
      \uic{\Gamma|X\vdash X,\Delta}{\AXL}
      \uic{\Gamma\vdash X,\Delta}{(*)}
      \uic{\Gamma\vdash X|\Delta}{\MU}
    \end{prooftree}
    
  \item[$(\ref{refutes})$] Suppose $(\Gamma,\Delta):X\Vdash$, i.e.,
    \begin{equation*}\tag{*}
      \forall(\Gamma',\Delta')\ge (\Gamma,\Delta),
      \left\{(\Gamma',\Delta') :\Vdash X \Tocom  \Gamma'\vdash\Delta'\right\}
    \end{equation*}
    We use $(\ast)$ to prove $\Gamma, X \vdash \Delta$ without
    introducing a cut from which the claim follows by the
    \MUT-rule. For this, we need to show $((\Gamma, X), \Delta) :
    \Vdash X$. Assume $(\Gamma'', \Delta'') \ge ((\Gamma,X),\Delta)$
    such that there is a cut-free proof for $\Gamma'' \mid X \vdash
    \Delta''$. Then by $(Contr_L)$, $\Gamma'' \vdash \Delta''$, that
    is, $(\Gamma'', \Delta'')$ is exploding.
  
  \end{enumerate}

  \paragraph{Base cases for $\top$ and $\bot$} Obvious.

  \paragraph{Induction case for implication}
  \begin{enumerate}
  \item[$(\ref{forces})$] Suppose $(\Gamma,\Delta): \Vdash A_1\to
    A_2$, i.e.,
    \begin{equation*}\tag{*}
      \forall(\Gamma',\Delta')\ge (\Gamma,\Delta),\{
      (\Gamma',\Delta'): A_1 \to A_2 \sVd\Tocom
      \Gamma'\vdash\Delta'\}
    \end{equation*}

    We use $(*)$ to prove $\Gamma, A_1 \vdash A_2, \Delta$ without
    introducing a cut from which the claim follows by the \MU~ and
    \IMPR~ rules. We need to show $((\Gamma, A_1), (A_2,\Delta)) : A_1
    \to A_2 \sVd$, i.e. $((\Gamma, A_1), (A_2,\Delta)) :\Vdash A_1$
    and $((\Gamma, A_1), (A_2,\Delta)) :A_2 \Vdash$. We show the first
    one. The second case is similar.\smallskip

    Assume $(\Gamma',\Delta') \ge ((\Gamma,A_1), (\Delta,A_2))$ such
    that $(\Gamma', \Delta') : A_1\sVd$. Using the induction
    hypothesis we get the following cut-free proof:
  \begin{prooftree}
    \axc{\Gamma' \mid A_1 \vdash \Delta'}
    \uic{\Gamma'\vdash \Delta'}{$(Contr_L)$}
  \end{prooftree}
  That is, $(\Gamma', \Delta')$ is exploding.

\item[$(\ref{refutes})$] Suppose $(\Gamma,\Delta):A_1\to A_2\Vdash$,
  i.e.,
  \begin{equation*}\tag{*}
    \forall(\Gamma',\Delta')\ge (\Gamma,\Delta),\{
    (\Gamma',\Delta'):\Vdash A_1 \to A_2 \Tocom
    \Gamma'\vdash\Delta'\}
  \end{equation*}

  We use $(*)$ to prove $\Gamma, A_1\to A_2 \vdash \Delta$ without
  introducing a cut from which the claim follows by the \MUT-rule. We
  need to show $((\Gamma,A_1\to A_2), \Delta) :\Vdash A_1 \to
  A_2$. Assume $(\Gamma'',\Delta'') \ge ((\Gamma, A_1\to A_2),
  \Delta)$ such that $(\Gamma'', \Delta'') \Vdash A_1$ and $(\Gamma'',
  \Delta'') : A_2 \Vdash$. Then, using the induction hypotheses we
  have the following cut-free proof:
  \begin{prooftree}
    \axc{\Gamma'' \vdash A_1 \mid \Delta''}
    \axc{\Gamma''\mid A_2\vdash \Delta''}
    \bic{\Gamma''\mid A_1\to A_2\vdash \Delta''}{$(\to_L)$}
    \uic{\Gamma''\vdash \Delta''}{$(Contr_L)$}
  \end{prooftree}
  That is, $(\Gamma'', \Delta'')$ is exploding.

\end{enumerate}

\paragraph{Induction case for $\vee$}
\begin{enumerate}
\item[$(\ref{forces})$] Suppose $(\Gamma,\Delta):\Vdash A_1\vee A_2$,
  i.e.,
  \begin{equation*}\tag{*}
    \forall (\Gamma',\Delta')\ge (\Gamma,\Delta),
    \{(\Gamma',\Delta') : A_1 \lor A_2 \sVd \Tocom
    (\Gamma',\Delta') \explodes\}
  \end{equation*}
  First we use $(\ast)$ to show $\Gamma \vdash A_1, A_2, A_1 \lor A_2,
  \Delta$ without introducing a cut. For this we set $\Gamma'= \Gamma$ and
  $\Delta' = A_1, A_2, A_1 \lor A_2, \Delta$, that is, we need to show
  $(\Gamma', \Delta') : A_i \Vdash$ for $i = 1, 2$. Assume $(\Gamma'',
  \Delta'') \ge (\Gamma', \Delta')$ such that $(\Gamma'', \Delta'')
  :\Vdash A_i$, then by induction hypotheses $\Gamma'' \vdash A_i \mid
  \Delta''$. Therefore, by $(Contr_R)$, $(\Gamma'', \Delta'')$ is
  exploding.
  
  Now we can prove the claim.
  \begin{prooftree}
    \axc{\Gamma\vdash A_2,A_1,A_1\vee A_2,\Delta}
    \uic{\Gamma\vdash A_2|A_1,A_1\vee A_2,\Delta}{\MU}
    \uic{\Gamma\vdash A_1\vee A_2|A_1,A_1\vee A_2,\Delta}{$(\vee^2_L)$}
    \uic{\Gamma\vdash A_1,A_1\vee A_2,\Delta}{$(Contr_R)$}
    \uic{\Gamma\vdash A_1|A_1\vee A_2,\Delta}{\MU}
    \uic{\Gamma\vdash A_1\vee A_2|A_1\vee A_2,\Delta}{$(\vee^1_L)$}
    \uic{\Gamma\vdash A_1\vee A_2,\Delta}{$(Contr_R)$}
    \uic{\Gamma\vdash A_1\vee A_2|\Delta}{\MU}
  \end{prooftree}

\item[$(\ref{refutes})$] The claim follows directly from the
  \DISJL-rule and the induction hypothesis because $(\Gamma,\Delta) :
  A_1 \lor A_2 \Vdash$ implies both $(\Gamma,\Delta) : A_1\Vdash$ and
  $(\Gamma,\Delta) : A_2 \Vdash$ by Lemma \ref{cbn-to-cbv}, which does
  not need to introduce new cuts.

\end{enumerate}

\paragraph{Induction case for $\wedge$}
\begin{enumerate}
\item[$(\ref{forces})$] The claim follows directly from the
  \CONJR-rule and the induction hypotheses because $(\Gamma,\Delta) :
  \Vdash A_1 \land A_2$ implies both $(\Gamma,\Delta) :\Vdash A_1$ and
  $(\Gamma,\Delta) :\Vdash A_2$, by Lemma \ref{forcing_comparison},
  which does not need to intruduce new cuts.

\item[$(\ref{refutes})$] Suppose $(\Gamma,\Delta): A_1\land A_2
  \Vdash$, i.e.,
  \begin{equation*}\tag{*}
    \forall (\Gamma',\Delta')\ge (\Gamma,\Delta),
    \{(\Gamma',\Delta') :\Vdash A_1 \land A_2  \Tocom
    (\Gamma',\Delta') \explodes\}
  \end{equation*}
  We use $(\ast)$ to show $\Gamma, A_1 \land A_2 \vdash \Delta$
  without introducing a cut from which the claim follows by the
  \MUT-rule.  By Lemma~\ref{forcing_comparison}, we need to show
  $((\Gamma,A_1\land A_2), \Delta) :\Vdash A_i$ for $i = 1, 2$. Assume
  $(\Gamma'', \Delta'') \ge ((\Gamma, A_1\land A_2), \Delta)$ such
  that $(\Gamma'',\Delta'') : A_i \sVd$. Using induction hypotheses we
  get the following cut-free proof:
  \begin{prooftree}
    \axc{\Gamma'' \mid A_i \vdash \Delta''}
    \uic{\Gamma'' \mid A_1 \land A_2 \vdash \Delta''}{$(\land_L^i)$}
    \uic{\Gamma'' \vdash \Delta''}{$(Contr_L)$}
  \end{prooftree}
  Therefore, $(\Gamma'', \Delta'')$ is exploding.
    
  \end{enumerate}

  \paragraph{Induction case for $\forall$}
  \begin{enumerate}
  \item[$(\ref{forces})$] Assume $(\Gamma,\Delta) :\Vdash \forall
    x. A(x)$. Then, by Lemma~\ref{forcing_comparison},
    $(\Gamma,\Delta) :\Vdash A(t)$ for all closed terms. In
    particular, we have $(\Gamma,\Delta) :\Vdash A(c)$ for some fresh
    constant $c$ which does not occur in $\Gamma,\Delta,A$. Using the
    induction hypothesis we get a cut-free proof of $\Gamma \vdash
    A(c) \mid \Delta$. By Lemma \ref{fresh}, this implies a cut-free
    proof of $\Gamma \vdash A(x) \mid \Delta$ for any fresh variable
    $x$, so the claim follows.

  \item[$(\ref{refutes})$] Suppose $(\Gamma,\Delta): \forall x. A(x)
    \Vdash$, i.e.,
    \begin{equation*}\tag{*}
      \forall (\Gamma',\Delta')\ge (\Gamma,\Delta),
      \{(\Gamma',\Delta') :\Vdash \forall x. A(x) \Tocom
      (\Gamma',\Delta') \explodes\}
    \end{equation*}
    We use $(\ast)$ to show $\Gamma, \forall x. A(x) \vdash \Delta$
    without introducing a cut from which the claim follows by the
    \MUT-rule, that is, we need to show $((\Gamma,\forall x. A(x)),
    \Delta) :\Vdash A(t)$ for any closed term $t$. Assume $(\Gamma'',
    \Delta'') \ge ((\Gamma,\forall x. A(x)), \Delta)$ such that
    $(\Gamma'',\Delta'') : A(t) \sVd$. Using the induction hypothesis we
    get the following cut-free proof:
  \begin{prooftree}
    \axc{\Gamma'' \mid A(t) \vdash \Delta''}
    \uic{\Gamma'' \mid \forall x. A(x) \vdash \Delta''}{$(\forall_L)$}
    \uic{\Gamma'' \vdash \Delta''}{$(Contr_L)$}
  \end{prooftree}
  Therefore, $(\Gamma'', \Delta'')$ is exploding.

  \end{enumerate}

  \paragraph{Induction case for $\exists$}
  \begin{enumerate}
  \item[$(\ref{forces})$] Suppose $(\Gamma,\Delta):\Vdash \exists
    x. A(x)$, i.e.,
  \begin{equation*}\tag{*}
    \forall (\Gamma',\Delta')\ge (\Gamma,\Delta),
    \{(\Gamma',\Delta') : \exists x. A(x) \sVd \Tocom
    (\Gamma',\Delta') \explodes\}
  \end{equation*}
  We use $(\ast)$ to show $\Gamma \vdash \exists x. A(x), \Delta$
  without introducing a cut from which the claim follows using the
  \MU-rule. We need to show $(\Gamma, (\Delta, \exists x. A(x))) :
  A(t) \Vdash$ for any closed term $t$.

  Assume $(\Gamma'', \Delta'') \ge (\Gamma, (\Delta, \exists x.A(x)))$
  such that $(\Gamma'', \Delta'') :\Vdash A(t)$. Using the induction
  hypothesis we get the following cut-free proof:
  \begin{prooftree}
    \axc{\Gamma'' \vdash A(t) \mid \Delta''}
    \uic{\Gamma'' \vdash \exists x. A(x) \mid \Delta''}{$(\exists_R)$}
    \uic{\Gamma'' \vdash \Delta''}{$(Contr_R)$}
  \end{prooftree}
  Therefore, $(\Gamma'', \Delta'')$ is exploding.

\item[$(\ref{refutes})$] Assume $(\Gamma,\Delta) : \exists x. A(x)
  \Vdash$, then $(\Gamma,\Delta) : \exists x. A(x) \sVd$ by
  Lemma~\ref{cbn-to-cbv}. That is, $(\Gamma,\Delta) : A(t)\Vdash$ for
  all closed terms. In particular, we have $(\Gamma,\Delta) : A(c)
  \Vdash$ for some fresh constant $c$ which does not occur in
  $\Gamma,\Delta,A$. Using induction hypotheses we have a cut-free
  proof of $\Gamma \mid A(c) \vdash \Delta$. By Lemma \ref{fresh},
  this implies a cut-free proof of $\Gamma \mid A(x) \vdash \Delta$
  for any fresh variable, so the claim follows.

\end{enumerate}

\end{proof}

\begin{corollary}\label{condition}
  For any sentence $A$ and contexts of sentences $\Gamma, \Delta$, the
  following hold in \Universal:
  \begin{enumerate}
  \item If $A \in \Gamma$ then $(\Gamma,\Delta):\Vdash A$.\label{nt1}

  \item If $B \in \Delta$ then $(\Gamma,\Delta): B \Vdash$.\label{nt2}
  \end{enumerate}
\end{corollary}

\begin{proof} 
  \begin{enumerate}
  \item Assume $A \in \Gamma$, $(\Gamma',\Delta') \ge (\Gamma,\Delta)$
    and $(\Gamma', \Delta') : A \sVd$. Then by Theorem
    \ref{Ucompleteness}, $\Gamma' \mid A \vdash \Delta'$, so we obtain
    a cut-free proof for $\Gamma' \vdash \Delta'$ using
    $(Contr_L)$. That is, $(\Gamma',\Delta')$ is exploding.

  \item Assume $B \in \Delta$, $(\Gamma',\Delta') \ge (\Gamma,\Delta)$
    and $(\Gamma', \Delta') :\Vdash B$. Then by Theorem
    \ref{Ucompleteness}, $\Gamma' \vdash B \mid \Delta'$, so we obtain
    a cut-free proof for $\Gamma' \vdash \Delta'$ using
    $(Contr_R)$. That is, $(\Gamma',\Delta')$ is exploding.
 \end{enumerate}
\end{proof}

\begin{corollary}[Completeness of Classical Logic]\label{usual_completeness}
  If in every Kripke model, at every possible world, the sentence $A$
  is forced whenever all the sentences of $\Gamma$ are forced and all
  the sentences of $\Delta$ are refuted, then there exists a cut-free
  derivation in {\em \LKMMT} of the sequent $\Gamma\vdash A|\Delta$.
\end{corollary}
\begin{proof}
  If the hypothesis holds for any Kripke model, so does it hold for
  \Universal. Theorem \ref{Ucompleteness} and Corollary
  \ref{condition} lead to the claim, since $(\Gamma,\Delta):\Vdash
  \Gamma$ and $(\Gamma,\Delta): \Delta \Vdash$.
\end{proof}

\begin{remark}\label{or-exists2}
  The following are false, even if reasoning classically.
\begin{itemize}
\item $w :\Vdash A \lor B$ implies $w:\Vdash A$ or $w :\Vdash B$.

\item $w:\Vdash \exists x.A(x)$ implies $w:\Vdash A(d)$ for some $d\in
  D(w)$.
\end{itemize}
Because of the completeness of classical logic with respect to the
universal model, the claims correspond to Disjunction property (DP)
and Explicit definability property (ED), respectively, which are in
general not true in classical logic.
\end{remark}

A constructive cut-free completeness theorem can also be used for
proof normalisation.

\begin{corollary}[Semantic Cut-Elimination]\label{nbe}
  For all contexts $\Gamma, \Delta$ of sentences, if there is a
  derivation of $\Gamma\vdash\Delta$, then there is a cut-free
  derivation of $\Gamma\vdash\Delta$.
\end{corollary}

\begin{proof}
  From the hypothesis $\Gamma\vdash\Delta$, the soundness theorem
  applied to \Universal\ gives us that there is indeed a cut-free
  derivation for $\Gamma\vdash\Delta$ because the world
  $(\Gamma,\Delta)$ forces all formulae of $\Gamma$ and refutes all
  formulae of $\Delta$ as shown in Corollary \ref{condition}.
\end{proof}

\section{Discussion, Related and Future Work}

\subsection{Normalisation by Evaluation}
The last corollary
is at the origin of our work, where we wanted to do a
normalisation-by-evaluation (NBE) proof for computational classical
logic.  The general idea of the NBE method is to use an ``evaluation''
(soundness) function from the object-language to a constructive
meta-language and then use a ``reification'' (completeness) function
from the meta-language back to the object-language. The interpretation
of the object-language inside the meta-language, that goes via
evaluation/soundness, is usually done using some form of Kripke
models.

So far, NBE has been used to show normalisation of various
intuitionistic proof systems
\cite{DBLP:conf/lics/BergerS91,DBLP:conf/popl/Danvy96,DBLP:conf/lics/AbelCD07,DBLP:conf/csl/Abel09,DBLP:journals/tcs/Okada02,sozeaua}
as well as purely computational calculi
\cite{DBLP:conf/fossacs/FilinskiR04}. One advantage of taking this
approach to that of studying a reduction relation for a proof calculus
for classical logic, explicitly as a rewrite system, is that one
circumvents both difficulties of rewrite systems and validating
equalities arising from $\eta$-conversion. For more details on these
difficulties the reader is referred to \cite{UrbanR06}, for classical
proof systems, and \cite{DBLP:journals/apal/FioreCB06} for
intuitionistic proof systems. Another advantage is that these kinds of
proofs manipulate finite structures only and avoid working with
saturated models as, for example, in \cite{Troelstra}. 

Note also that, although as output from the NBE algorithm we get a
$\beta$-reduced $\eta$-long normal form, we proved a weak NBE result,
as we did not prove that the output can be obtained from the input by
a number of rewrite steps, as it is done in \cite{cCoquand93}.

\subsection{Dual Notion of Model}

Thanks to the symmetry of the {\LKMMT} rules for left-distinguished and
right-distinguished formulae, it is possible to define a dual notion
of model in which:
\begin{itemize}
\item ``strong \emph{forcing}'' is taken as primitive and
  ``refutation'' and non-strong ``forcing'' are defined from it by
  orthogonality like in Definition~\ref{def-composite},
\item for the universal model, strong forcing is defined as cut-free
  provability of \emph{right}-distinguished formulae (instead of
  left-distinguished ones for strong refutation),
\end{itemize}
and prove, completely analogously to the proofs presented in this
paper, that we have the same soundness and completeness theorems
holding. 

The reader interested in the computational behaviour of the
completeness theorem, should look at its partial Coq
formalisation\cite{formalisation}. From that work it follows that the
NBE theorem computes the normal forms of proofs in call-by-name
discipline. We mention this work because we would like to conjecture
that the presented classical Kripke model always gives rise to
call-by-name behaviour for proof normalisation, while the dual notion
gives rise to call-by-value behaviour. As one of the referees
remarked, there is a variety of different strategies for doing proof
normalisation, of which call-by-name and call-by-value are the
simplest ones to describe, but also the most standard ones. For a
general study of cut-elimination strategies that are more complex than
call-by-name and call-by-value, the reader is referred to
\cite{danosjs}.

\subsection{Using Intuitionistic Kripke Models on Doubly-Negated
  Formulae}
Although one can define a double-negation interpretation $A^*$ of
formulae and use intuitionistic Kripke models and an intuitionistic
completeness theorem to obtain a normalisation result, one would have
to pass through the chain of inferences
\[
\vdash_c A \Longrightarrow ~\vdash_i A^* \Longrightarrow ~\Vdash_i A^*
\Longrightarrow ~\vdash^{\mathit{nf}}_i A^* \Longrightarrow
~\vdash^{\mathit{nf}}_c A
\]
where ``i'' stands for ``intuitionistic'', ``c'' for ``classical'' and
``nf'' for ``in normal form'', in which how to do the last inference
is not obvious. We consider that to be a \emph{detour} since we can
prove, simply, the chain of inferences
\[
\vdash_c A \Longrightarrow ~\Vdash_c A \Longrightarrow
~\vdash^{\mathit{nf}}_c A
\]

The interest in having a direct-style semantics for classical logic is
the same as the interest in having a proof calculus for classical
logic instead of restricting oneself to an intuitionistic calculus and
working with doubly-negated formulae; or, in the theory of programming
languages, to having a separate constant \emph{call-cc} instead of
writing all programs in continuation-passing style.

Avigad shows in \cite{DBLP:journals/jlp/Avigad01} how classical
cut-elimination is a special case of intuitionistic one, work which
resembles the first chain of inferences of this subsection. However,
his work is specialised to ``negative'' formulae, that is, it is not
clear how to extend it to formulae that use $\vee$ and $\exists$.

Finally, we remark that an interpretation through intuitionistic
Kripke models and a double-negation interpretation would have to be
done in Kripke models with exploding nodes, because of the
meta-mathematical results from \cite{Kreisel62, McCarty94, McCarty02}.

\subsection{Boolean vs. Kripke Semantics for Classical Logic}

We compare Boolean and Kripke semantics in a constructive setting,
based on our own observations (which we hope to submit for publication
soon) and based on a strand of works in mathematical logic from the
1960s.

\paragraph{Computational Behaviour} The only known constructive
completeness proof of classical logic with respect to Boolean models
is the one of Krivine\cite{Krivine96}, who used a double-negation
interpretation to translate G\"odel's original proof. Krivine's proof
was later reworked by Berardi and Valentini \cite{BerardiV04} to show
that its main ingredient is a constructive version of the ultra-filter
theorem for countable Boolean algebras. This theorem, however,
crucially relies on an enumeration of the members of the algebra (the
formulae).

In the work we mentioned as yet to be put into words, a formalisation
in constructive type theory of the proof of Berardi and Valentini, we
saw that, as a consequence of relying on the linear order, the
reduction relation for proof-terms corresponding to implicative
formulae is not $\beta$-reduction, but an ad hoc reduction relation
which depends on the particular way one defines the linear order
(enumeration of formulae). As a consequence, there is no clear notion
of normal form suggested by the ad hoc reduction relation. The
cut-free completeness theorem given in this paper, however, gives rise
to a normalisation algorithm which respects the $\beta$-reduction
relation of the object-language, when the Kripke models are
interpreted in a type theory which is based on $\beta$-reduction
itself. 

\paragraph{Expressiveness} We think of classical Kripke model validity
as being more expressive, i.e. containing more information, than
Boolean model validity. That is indicated by the presented
completeness theorem which is both simpler than (constructive)
completeness theorems for Boolean models, and manipulates finite
structures directly, instead of relying on structures built up by an
infinite saturation process.

Also, only after submitting the first version of the present text, we
became aware of the work done in the 1960s on using Kripke models to
do model theory of classical logic \cite{Fitting69}. Although
conducted in a \emph{classical} meta-language, the work indicates that
it is possible to use Kripke models to express elegantly some
cumbersome constructions of model theory, like set theoretic forcing
\cite{Dahn79,Fitting69}. Indeed, the connection between the two had
been spotted already by Kripke \cite{kripke63} and hence the term
``forcing'' appeared in Kripke semantics. We hope that looking at
those kind of constructions inside Kripke models, but this time inside
a \emph{constructive} meta-language, might be an interesting venue to
finding out the constructive content of techniques of classical model
theory.

In this respect, our work can also be seen as a contribution to the
field of constructive model theory of classical logic.

\bibliographystyle{plain} 
\bibliography{classical_kripke_rev3}
\end{document}